\documentclass[a4paper,11pt]{article}

\usepackage[a4paper]{geometry}

\usepackage{authblk}

\usepackage{amsmath}
\usepackage{amsthm}
\usepackage{amssymb}
\usepackage{amsfonts}
\usepackage{mathtools}
\usepackage{bm}

\usepackage{graphicx}
\usepackage{tikz}
\usetikzlibrary{cd,backgrounds,positioning,trees,shapes,arrows,patterns,topaths,calc}

\usepackage[
  backend=biber,
  style=numeric,
  maxcitenames=9,
  maxbibnames=99,
  abbreviate=false,
  eprint=false,
  doi=true,
  giveninits=true,
  casechanger=expl3,
  isbn=false
]{biblatex}
\DeclareFieldFormat{doi}{\url{https://doi.org/#1}}

\usepackage{hyperref}
\hypersetup{
  hidelinks,
  linktoc = all,
  breaklinks = true,
}
\usepackage[nameinlink]{cleveref}

\newtheorem{theorem}{Theorem}
\newtheorem{lemma}[theorem]{Lemma}

\newtheorem{fact}[theorem]{Fact}

\theoremstyle{definition}
\newtheorem{definition}[theorem]{Definition}
\newtheorem{remark}[theorem]{Remark}

\theoremstyle{remark}

\numberwithin{theorem}{section}

\AddToHook{env/lemma/begin}{\crefalias{theorem}{lemma}}
\AddToHook{env/corollary/begin}{\crefalias{theorem}{corollary}}
\AddToHook{env/fact/begin}{\crefalias{theorem}{fact}}
\AddToHook{env/conjecture/begin}{\crefalias{theorem}{conjecture}}
\AddToHook{env/question/begin}{\crefalias{theorem}{question}}
\AddToHook{env/proposition/begin}{\crefalias{theorem}{proposition}}

\AddToHook{env/definition/begin}{\crefalias{theorem}{definition}}
\AddToHook{env/remark/begin}{\crefalias{theorem}{remark}}
\AddToHook{env/example/begin}{\crefalias{theorem}{example}}
\AddToHook{env/notation/begin}{\crefalias{theorem}{notation}}

\AddToHook{env/claim/begin}{\crefalias{theorem}{claim}}

\crefname{lemma}{lemma}{lemmas}
\crefformat{lemma}{#2lemma~#1#3} 
\Crefformat{lemma}{#2Lemma~#1#3} 

\crefname{corollary}{corollary}{corollaries}
\crefformat{corollary}{#2corollary~#1#3}
\Crefformat{corollary}{#2Corollary~#1#3}

\crefname{fact}{fact}{facts}
\crefformat{fact}{#2fact~#1#3}
\Crefformat{fact}{#2Fact~#1#3}

\crefname{conjecture}{conjecture}{conjectures}
\crefformat{conjecture}{#2conjecture~#1#3}
\Crefformat{conjecture}{#2Conjecture~#1#3}

\crefname{question}{question}{questions}
\crefformat{question}{#2question~#1#3}
\Crefformat{question}{#2Question~#1#3}

\crefname{proposition}{proposition}{propositions}
\crefformat{proposition}{#2proposition~#1#3} 
\Crefformat{proposition}{#2Proposition~#1#3} 

\crefname{definition}{definition}{definitions}
\crefformat{definition}{#2definition~#1#3}
\Crefformat{definition}{#2Definition~#1#3}

\crefname{remark}{remark}{remarks}
\crefformat{remark}{#2remark~#1#3}
\Crefformat{remark}{#2Remark~#1#3}

\crefname{example}{example}{examples}
\crefformat{example}{#2example~#1#3}
\Crefformat{example}{#2Example~#1#3}

\crefname{notation}{notation}{notations}
\crefformat{notation}{#2notation~#1#3}
\Crefformat{notation}{#2Notation~#1#3}

\crefname{claim}{claim}{claims}
\crefformat{claim}{#2claim~#1#3}
\Crefformat{claim}{#2Claim~#1#3}

\renewcommand{\sf}[1]{\mathsf{#1}}

\newcommand{\fr}[1]{\mathcal{#1}}
    \newcommand{\F}{\fr{F}}

\newcommand{\class}[1]{\mathcal{#1}}
    \newcommand{\V}{\class{V}}

\renewcommand{\L}{\mathcal{L}}

\newcommand{\FFr}{\sf{FFr}}

\newcommand{\fmp}{\mathrm{fmp}}

\newcommand{\pow}[1]{\mathcal{P} (#1)}

\newcommand{\logic}[1]{\mathsf{#1}}
    \newcommand{\Prop}{\logic{Prop}}
    \newcommand{\Fml}{\logic{Fml}}

    \newcommand{\Th}{\logic{Th}}
    \newcommand{\T}{\logic{T}}
    
    \newcommand{\Eq}{\logic{Eq}}

\newcommand{\NExt}[1]{\mathop{\mathsf{NExt}}{#1}}
\newcommand{\Ext}[1]{\mathop{\mathsf{Ext}}{#1}}

\renewcommand{\sf}[1]{\mathsf{#1}}

\newcommand{\Val}{\mathsf{Val}}

\renewcommand{\phi}{\varphi}

\newcommand{\conti}{2^{\aleph_0}}

\newcommand{\Lor}{\bigvee}
\newcommand{\Land}{\bigwedge}

\newcommand{\tab}{\sf{tab}}

\newenvironment{acknowledgements}{%
  \begin{abstract}
}{%
  \end{abstract}
}

\newenvironment{funding}{%
  \begin{abstract}
}{%
  \end{abstract}
}

\usepackage{orcidlink}

\begin{document}
\title{The Cardinalities of Intervals of Equational \\ Theories and Logics}

\author[1]{Juan P. Aguilera \orcidlink{0000-0002-2768-6714}}
\author[2]{Nick Bezhanishvili \orcidlink{0009-0005-6692-5051}}
\author[2]{Tenyo Takahashi \orcidlink{0009-0003-8147-7045}}

\affil[1]{Institute of Discrete Mathematics and Geometry, TU Wien}
\affil[2]{Institute for Logic, Language and Computation, University of Amsterdam}

\date{}
\maketitle

\begin{abstract}
We study the cardinality of classes of equational theories (varieties) and logics by applying descriptive set theory. We affirmatively solve open problems raised by Jackson and Lee [Trans. Am. Math. Soc. 370 (2018), pp. 4785-4812] regarding the cardinalities of subvariety lattices, and by Bezhanishvili et al. [J. Math. Log. (2025), in press] regarding the degrees of the finite model property (fmp). By coding equations and formulas by natural numbers, and theories and logics by real numbers, we examine their position in the Borel hierarchy. We prove that every interval of equational theories in a countable language corresponds to a $\bm{\Pi}^0_1$ set, and every fmp span of a normal modal logic to a $\bm{\Pi}^0_2$ set. It follows that they have cardinality either $\leq \aleph_0$ or $2^{\aleph_0}$, provably in ZFC. In the same manner, we observe that the set of pretabular extensions of a tense logic is a $\bm{\Pi}^0_2$ set, so its cardinality is either $\leq \aleph_0$ or $2^{\aleph_0}$. We also point out a negative solution to another open problem raised by Jackson and Lee, \emph{op. cit.}, regarding the existence of independent systems, which relies on Je{\v z}ek et al. [Bull. Aust. Math. Soc. 42 (1990), pp. 57-70].
\end{abstract}

\section{Introduction}
An important theme in the studies of universal algebra and non-classical logics involves determining the cardinality of classes of equational theories (equivalently, varieties) and logics. A \emph{variety} is an equationally definable class of algebras \cite{ACourseInUniversalAlgebra1981}. Finite monoids that generate varieties with $\conti$ many subvarieties were constructed in \cite{jacksonInterpretingGraphColorability2006}, and later a finite monoid that generates a variety with exactly $\aleph_0$ subvarieties was found in \cite{jacksonMonoidVarietiesExtreme2018}.  
In \cite{jacksonMonoidVarietiesExtreme2018}, Jackson and Lee raised the question of whether the subvariety lattice of every variety has cardinality either $\leq \aleph_0$ or $2^{\aleph_0}$ \cite[Question 6.4 (ii)]{jacksonMonoidVarietiesExtreme2018}. 

In the logical context, Jankov \cite{jankov1968} showed that there are $\conti$ many distinct superintuitionistic logics, and consequently $\conti$ many modal logics in extensions of  $\mathsf{K}$, $\mathsf{K4}$, $\mathsf{S4}$, etc.. This laid a foundation for the study of cardinalities of classes of modal and superintuitionistic logics. 
Fine \cite{Fine71} showed that every normal extension of the modal logic $\mathsf{S4.3}$ is finitely axiomatizable. As a consequence, he obtained that $\mathsf{S4.3}$ has countably many normal extensions. 
For more information on cardinalities of intervals of modal and superintuitionistic logics, we refer to \cite{czModalLogic1997}.

For a superintuitionistic or modal logic $L$, the \emph{degree of fmp} of $L$ refers to the cardinality of the \emph{fmp span} of $L$, that is, the set of logics that define the same class of finite Kripke frames as $L$. Bezhanishvili et al.~\cite{bezhanishvili2023degreesfinitemodelproperty} proved that every cardinal $\kappa \leq \aleph_0$ and $\kappa = \conti$ is realized as the degree of fmp for some superintuitionistic and transitive modal logic. They also asked the question whether it is provable without the Continuum Hypothesis (CH) that the degree of fmp can (only) be either $\leq \aleph_0$ or $2^{\aleph_0}$ \cite[Section 8 (1)]{bezhanishvili2023degreesfinitemodelproperty}. This question is closely connected to the one posed by Jackson and Lee. A logic is \emph{tabular} if it is characterized by a single finite Kripke frame, and is \emph{pretabular} if it is not tabular while all its consistent extensions are tabular. Recently, Chen \cite{chen2025pretabulartenselogicss4t} showed that for every cardinal $\kappa \leq \aleph_0$ and $\kappa = \conti$, there is a tense logic that has exactly $\kappa$ pretabular extensions. Even though not explicitly asked in the paper, the question of the cardinals $\aleph_0 < \kappa < \conti$ naturally arises in the same spirit as the above two open questions. 

The common theme of these questions is the cardinality of classes of equational theories and logics between countable and continuum, hence they are inherently set-theoretic. It is not to be expected that standard tool kits in the field of universal algebra or modal and superintuitionistic logics could be useful for answering these kinds of questions, as they fall in the scope of ZFC, of which CH is independent. An explicit construction of uncountably many equational theories or logics typically results in constructing $\conti$-many of such equational theories or logics. 

In this paper, we apply descriptive set theory via coding to address the aforementioned three questions. The coding method aligns with the standard technique of G\"odel numbers in computability theory. The idea behind the proofs is to effectively code equations (resp. formulas) by natural numbers, so that varieties (resp. logics) correspond to subsets of $\omega$, which in turn are identified with elements of the Cantor space $2^\omega$, namely, real numbers. Thus, instead of varieties or logics, we can count the corresponding sets of reals. That descriptive set theory applies in this context is no surprise, as it is exactly the theory of sets of reals. In descriptive set theory, sets of reals are classified in several hierarchies in terms of their ``complexity'' (see \Cref{Sec 2}), and their properties are studied with respect to their complexity. The main result that we rely on in this paper is that Borel sets have cardinality either $\leq \aleph_0$ or $2^{\aleph_0}$. 

Therefore, all we have to do is verify that the sets of reals corresponding to the sets of equational theories and logics in question are Borel. Specifically, we show that an interval of equational theories in a countable language corresponds to a $\bm{\Pi}^0_1$ set, and the fmp span of a modal logic to a $\bm{\Pi}^0_2$ set. This implies that their cardinality is either $\leq \aleph_0$ or $\conti$, provably in ZFC, answering the open questions by Jackson and Lee \cite[Question 6.4 (ii)]{jacksonMonoidVarietiesExtreme2018} and by Bezhanishvili et al.~\cite[Section 8 (1)]{bezhanishvili2023degreesfinitemodelproperty} in the positive. On the other hand, we also point out that it follows from \cite{jezekEquationalTheoriesSemilattices1990} that there is an equational theory such that its extensions form a continuum chain, which thus admits no independent system. This example provides a negative solution to \cite[Question 6.4 (i)]{jacksonMonoidVarietiesExtreme2018} in its stated generality, although it does not settle the question for monoids. Moreover, it implies that the cardinality result of intervals of equational theories cannot be shown by constructing independent systems, which explains why we need to employ the toolkit of descriptive set theory. We also observe that even though the definition of pretabularity is more involved and contains a universal quantification over logics, a more sophisticated characterization shows that the corresponding sets of reals are $\bm{\Pi}^0_2$. This yields a similar cardinality result and completes the cardinality characterization of pretabular tense logics in \cite{chen2025pretabulartenselogicss4t}.

The paper is organized as follows. \Cref{Sec 2} recalls preliminaries on descriptive set theory, universal algebra, and modal logic. The main results are proved in \Cref{Sec 3}. \Cref{Sec 4} concludes the paper with open questions.

\section{Preliminaries} \label{Sec 2}

We work in ZFC throughout the paper.

\subsection{Descriptive set theory}
We assume familiarity with basic definitions of descriptive set theory. We use \cite[Section 12]{kanamori2008higher} and \cite[Chapters 11 and 25]{jechSetTheory2003} as our main references for descriptive set theory. 

We will work with subsets of the Cantor space $2^\omega$, elements of which are called \emph{reals}. Reals are naturally identified with subsets of $\omega$. For $a \in 2^\omega$ and $m \in \omega$, the descriptions $m \in a$ and $m \notin a$ mean $a(m) = 1$ and $a(m) = 0$ respectively. The \emph{Borel hierarchy} ($\bm{\Sigma}^0_\alpha$, $\bm{\Pi}^0_\alpha$ for $0 < \alpha < \omega_1$) and the \emph{arithmetical hierarchy} ($\Sigma^0_n$, $\Pi^0_n$ for $n > 0$) are defined as usual. For a real $a \in 2^\omega$, a subset $A \subseteq 2^\omega$ is $\Sigma^0_n(a)$ if there is a $\Sigma^0_n$ formula $\phi(x)$ of the form $\exists m_1 \forall m_2 \cdots Q m_n R$ where all quantifiers range over $\omega$ and $R$ is a recursive predicate with the parameter $a$, such that $b \in A$ iff $\phi(b)$ holds for any $b \in 2^\omega$. In such a case, we say that $\phi$ \emph{defines} $A$. Similarly, a subset $A \subseteq 2^\omega$ is $\Pi^0_n(a)$ if $A$ is defined by a $\Pi^0_n$ formula with the parameter $a$. For example, the singleton set $\{a\}$ for a real $a \in 2^\omega$ is $\Pi^0_1(a)$ because it is defined by the formula $\phi(x) = \forall m (m \in x \leftrightarrow m \in a)$. We omit the precise definition of $\bm{\Sigma}^0_\alpha$ and $\bm{\Pi}^0_\alpha$, but note the following useful fact: for $n > 0$, a set $A \subseteq 2^\omega$ is $\bm{\Sigma}^0_n$ iff $A$ is $\Sigma^0_n(a)$ for some $a \in 2^\omega$, and similarly for $\bm{\Pi}^0_n$.

It is well-known, provably in ZFC, that Borel (even $\bm{\Sigma}^1_1$) sets have the \emph{perfect set property} -- thus the cardinality of a Borel set is either $\leq \aleph_0$ or $\conti$. We will exploit this fact to obtain the main results on the cardinality of classes of varieties, equational theories, and logics.

\begin{fact} \label{fact Borel}
    Every Borel set has cardinality either $\leq \aleph_0$ or $\conti$.
\end{fact}

\subsection{Equational theories and logics}
We refer to \cite{ACourseInUniversalAlgebra1981,UniversalAlgebraFundamentals2011} for basic facts of universal algebra.
A \emph{language}, or \emph{signature}, is a set of function symbols with each arity specified. Throughout the paper, we assume that the language is countable. Given a language and a set $X$ of variables, we define \emph{terms} over $X$ in the language in the standard way. An \emph{equation}, or \emph{identity}, over $X$ is an expression of the form $s \approx t$ where $s$ and $t$ are terms over $X$. Following the convention, we assume $X$ to be a countably infinite set and omit it if no confusion arises. A set $\Phi$ of equations is an \emph{equational theory} if $\Phi$ is the set of equations valid in a variety $\V$, where a \emph{variety} is a class of algebras closed under homomorphic images, subalgebras, and products. We can take the following syntactic characterization of equational theories as our working definition of equational theories in this paper (see, e.g., \cite[Chapter 2, Definition 14.16 and Theorem 14.17]{ACourseInUniversalAlgebra1981}).

\begin{definition} \label{2: Def replacement substitution} \leavevmode
    \begin{enumerate}
        \item For a term $t$ and an equation $s \approx s'$, an equation $t \approx t'$ is a \emph{replacement} instance of $t$ and $s \approx s'$ if $t'$ is the result of replacing an occurrence of $s$ in $t$ by $s'$.
        \item For an equation $s \approx s'$ and a tuple of terms $(t_1, \dots, t_n)$, the \emph{substitution} instance of $s \approx s'$ and $(t_1, \dots, t_n)$ is the resulting equation by simultaneously replacing every occurrence of each variable $x_i$ in $s \approx s'$ by $t_i$.
    \end{enumerate}
\end{definition}

\begin{theorem} \label{2: Thm equational theory}
    Let $\Phi$ be a set of equations. Then $\Th(\Phi)$, the least equational theory containing $\Phi$, is the least set of equations containing $\Phi$ such that: 
    \begin{enumerate}
            \item $s \approx s \in \Th(\Phi) \text{ for } s \in \T, $
            \item $s \approx t \in \Th(\Phi)$ implies $t \approx s \in \Th(\Phi),$
            \item $s \approx t,~ t \approx u \in \Th(\Phi)$ implies $s \approx u \in   \Th(\Phi),$
            \item $\Th(\Phi)$ is closed under replacement,
            \item $ \Th(\Phi)$ is closed under substitution.
        \end{enumerate}
        
\end{theorem}

We also recall the basics of modal logic. We refer to \cite{czModalLogic1997,blackburnModalLogic2001} for a comprehensive introduction to this subject. Modal \emph{formulas} are defined by the following syntax, where $\Prop$ is a countable set of propositional variables:
\begin{align*}
    \phi ::= p \:|\: \bot \:|\: \phi \land \psi \:|\: \phi \to \psi \:|\: \Box \phi, \quad p \in \Prop
\end{align*}
A \emph{normal modal logic} $L$ is a set of formulas that contains all the propositional tautologies and the K axiom $\Box (p \to q) \to \Box p \to \Box q$, and is closed under modus ponens, necessitation, and uniform substitution. A logic $L$ is a \emph{normal extension}, or simply an \emph{extension}, of another logic $L_0$ if $L_0 \subseteq L$. For any logic $L_0$, the extensions of $L_0$ form a complete lattice, denoted $\NExt{L_0}$, where the order is the subset relation $\subseteq$. \emph{Tense logics} are defined similarly with two boxes $\Box_0$ and $\Box_1$. Tense logics also contain two extra axioms $p \to \Box_0 \lnot \Box_1 \lnot p$ and $p \to \Box_1 \lnot \Box_0 \lnot p$. We call normal modal logics and tense logics \emph{logics} if no confusion arises.

A \emph{Kripke frame} $\F$ is a pair $(F, R)$ of a nonempty set $F$ and a binary relation $R \subseteq F \times F$. A \emph{valuation} on $\F$ is a function $V: \Prop \to \pow{F}$. This generalizes to all formulas as follows: Boolean connectives are interpreted by Boolean operations on $\pow{F}$, and $V(\Box \phi) = \{x \in F: \forall y (Rxy \to y \in V(\phi))\}$. We say that $\F$ \emph{validates} $\phi$, written $\F \models \phi$, if $V(\phi) = F$ for any valuation $V$ on $\F$. Kripke semantics for tense logics is defined analogously, where the converse relation of $R$ is used to interpret the second modality.

\section{Main results} \label{Sec 3}

The core idea behind the proofs in this section is \emph{coding}. If the language $\L$ is countable, then there are only countably many equations. These equations can be effectively coded by natural numbers relative to an oracle for $\L$, corresponding to the arity map of the function symbols in $\L$. For instance, one may employ a standard G\"odel numbering based on prime factorization or a pairing function. Then, the set $\sf{Eq} = \{i \in \omega: \text{$i$ is a code of an equation in $\L$}\}$ is recursive in $\L$. Let $\phi_i$ denote the equation with code $i$. Equational theories $\Phi$ then correspond to particular kinds of elements $A \in 2^\omega$, i.e., infinite $0$-$1$ sequences, in such a way that $A(i) = 1$ iff $\phi_i \in \Phi$. These will be the reals coding set of sentences satisfying some basic properties, such as the inclusion of all identities $s \approx s$ and closure under substitutions.
In that case, we also call $A$ the code of $\Phi$. We will be loose on the distinction between an equational theory and its code if there is no confusion. 
As every equational theory has a unique code, instead of counting equational theories, we can count elements of $2^\omega$, i.e., reals, where descriptive set theory applies. If we can show that the set of reals corresponding to a class of equational theories is in a specific complexity class, such as Borel, we can draw from descriptive set theory consequences on the cardinality of that class of equational theories. Since in modal logic the language is finite, it falls under a special case of the above observation.

We start with intervals of equational theories. Let $\T$ and $\Eq$ respectively be the sets of codes of all terms and equations. Let $\Phi_0$ and $\Phi_1$ be equational theories. Recall that the \emph{interval} between $\Phi_0$ and $\Phi_1$ is the set 
\[[\Phi_1, \Phi_2] = \{\Phi: \text{$\Phi$ is an equational theory such that } \Phi_1\subseteq \Phi \subseteq \Phi_2\}.\]
An interval in this sense may not be linearly ordered. Since each equational theory corresponds to a real, we can view an interval of equational theories as a set of reals. So, it is meaningful to talk about the complexity hierarchy of an interval of equational theories. Recall that for $A, B \subseteq \omega$, the \emph{join} $A \oplus B \subseteq \omega$ is the set 
\[\{2n: n \in A\} \cup \{2n+1: n \in B\}.\]
Intuitively, using the parameter $A \oplus B$ amounts to using the parameters $A$ and $B$.

\begin{lemma} \label{Lem interval Pi01}
    The set $[\Phi_0, \Phi_1]$ is $\Pi^0_1(\Phi_0 \oplus \Phi_1)$. 
\end{lemma}

\begin{proof}
    A set $\Phi \subseteq \omega$ is in $[\Phi_0, \Phi_1]$ iff each element in $\Phi$ indeed codes an equation, $\Phi$ is an equational theory, and $\Phi_1\subseteq \Phi \subseteq \Phi_2$. A set of equations is an equational theory iff it satisfies the five conditions in \Cref{2: Thm equational theory}. Thus, $\Phi \in [\Phi_0, \Phi_1]$ iff it satisfies all the following conditions:
    \begin{enumerate}
        \item $\Phi \subseteq \Eq$,
        \item $s \approx s \in \Phi \text{ for } s \in \T, $
        \item $s \approx t \in \Phi \Rightarrow t \approx s \in \Phi,$
        \item $s \approx t,~ t \approx u \in \Phi \Rightarrow s \approx u \in \Phi,$
        \item $\Phi \text{ is closed under replacement},$
        \item $\Phi \text{ is closed under substitution},$
        \item $\Phi_0 \subseteq \Phi \subseteq \Phi_1$.
    \end{enumerate}
    We show that all these conditions can be expressed by $\Pi^0_1$ formulas, with the parameters $\Phi_0$ and $\Phi_1$. We only address items (5) and (7). The others can be verified in a similar manner. First, observe that $\T$ and $\Eq$ are recursive in $\Phi_0$, since $\Phi_0$ serves as an oracle that takes any natural number as input (the same applies to $\Phi_1$). For instance, $t$ is a well-founded term iff $t \approx t \in \Phi_0$. Also, it follows from \Cref{2: Def replacement substitution} that the ternary relation ``$\phi'$ is a replacement instance of $t$ and $\phi$'' is recursive. So, using a recursive predicate $\sf{Rep}$, item (5) can be expressed by the $\Pi^0_1$ formula with the parameter $\Phi_0$ (hidden in $\T$):
    \[\forall i \forall j \forall k~(j \in \T \land k \in \Phi \land \sf{Rep}(i, j, k) \to i \in \Phi).\]
    Item (7) can be expressed by the $\Pi^0_1$ formula with parameters $\Phi_0$ and $\Phi_1$:
    \[\forall i~ [(i \in \Phi_0 \to i \in \Phi) \land (i \in \Phi \to i \in \Phi_1)].\]
    Thus, the set $[\Phi_0, \Phi_1]$ can be defined by a $\Pi^0_1$ formula with parameters $\Phi_0$ and $\Phi_1$, hence it is $\Pi^0_1(\Phi_0 \oplus \Phi_1)$.
\end{proof}

\begin{lemma} \label{Lem interval}
    For any equational theories $\Phi_0$ and $\Phi_1$, the interval $[\Phi_0, \Phi_1]$ has cardinality $\leq \aleph_0$ or $2^{\aleph_0}$. 
\end{lemma}

\begin{proof}
    The interval $[\Phi_0, \Phi_1]$ is $\Pi^0_1(\Phi_0 \oplus \Phi_1)$ by \Cref{Lem interval Pi01}, so it is $\bm{\Pi}^0_1$, and hence Borel. Thus, it has cardinality $\leq \aleph_0$ or $2^{\aleph_0}$ by \Cref{fact Borel}.
\end{proof}

\begin{lemma} \label{Lem eq th}
    Every equational theory has $\leq \aleph_0$ or $2^{\aleph_0}$ many extensions. 
\end{lemma}

\begin{proof}
    This follows from \Cref{Lem interval} by taking $\Phi_1$ to be the set of all equations.
\end{proof}

\begin{theorem}\label{TheoremVarieties}
    Every variety in a countable language has $\leq \aleph_0$ or $2^{\aleph_0}$ many subvarieties. 
\end{theorem}

\begin{proof}
    By Birkhoff's theorem (see, e.g., \cite[Theorem 11.9]{ACourseInUniversalAlgebra1981} and \cite[Theorem 4.11]{UniversalAlgebraFundamentals2011}) that a class of algebras is a variety iff it is equationally definable, there is a one-to-one correspondence between varieties and equational theories. Thus, we can reformulate \Cref{Lem interval} and \Cref{Lem eq th} in terms of varieties: if the language is countable, then for any varieties $\V_0$ and $\V_1$, the interval $[\V_0, \V_1]$ has cardinality $\leq \aleph_0$ or $2^{\aleph_0}$; in particular, every variety has $\leq \aleph_0$ or $2^{\aleph_0}$ many subvarieties. 
\end{proof}

This answers the open question \cite[Question 6.4 (ii)]{jacksonMonoidVarietiesExtreme2018} in the affirmative. 

\begin{remark} \label{remark open 1}
Jackson and Lee also raised a second problem: Does every variety with uncountably many subvarieties have an independent system extending its equational theory \cite[Question 6.4 (i)]{jacksonMonoidVarietiesExtreme2018}? An \emph{independent system} extending an equational theory $\Phi$ is an infinite set $\Psi$ of equations such that every subset $\Psi' \subseteq \Psi$ induces a distinct equational theory $\Th(\Phi \cup \Psi')$. Note that a positive answer to this question would also yield the conclusion of Theorem \ref{TheoremVarieties}, as we can construct $\conti$ many equational theories from $\conti$ many subsets of $\Psi$. However, we provide a negative answer to this question as follows.

We refer to \cite{UniversalAlgebraFundamentals2011, ACourseInUniversalAlgebra1981} for the necessary background on universal algebra. It follows from \cite[Theorem 3.1]{jezekEquationalTheoriesSemilattices1990} that every completely distributive algebraic lattice with compact top is isomorphic to the lattice of extensions of an equational theory. Let $S = \mathbb{Q} \cap [0, 1]$ and $L$ be the set of all downsets in $S$ ordered by inclusion. Then $L$ is a complete chain with size $\conti$, in particular, a completely distributive lattice, where $\Lor = \bigcup$ and $\Land = \bigcap$. If, for a rational $q \in \mathbb{Q}$, the principal downset ${\downarrow} q \subseteq \bigcup_{i \in I} D_i$ for some $\{D_i\}_{i \in I} \subseteq L$, then $q \in D_i$ for some $i \in I$, and thus ${\downarrow} q \subseteq D_i$. So, each principal downset ${\downarrow} q$ for $q \in \mathbb{Q}$ is a compact element in $L$. In particular, the top element ${\downarrow} 1$ is compact. Moreover, every downset $D \in L$ is the union $\bigcup_{q \in D} {\downarrow} q$ of compact elements below $D$, so $L$ is an algebraic lattice. Thus, by \cite[Theorem 3.1]{jezekEquationalTheoriesSemilattices1990}, there is an equational theory $T$ such that $\Ext T$ is isomorphic to $L$. Then, $|\Ext T| = |L| = \conti$, while there are no equations $\phi, \psi$ such that $\Th(T + \phi)$ and $\Th(T + \psi)$ are incomparable, so there are no equations $\phi, \psi$ such that $\Th(T + \phi)$, $\Th(T + \psi)$, and $\Th(T + \{\phi, \psi\})$ are pairwise distinct, hence there is no independent system extending $T$. This yields a negative answer to \cite[Question 6.4 (i)]{jacksonMonoidVarietiesExtreme2018}. However, it follows from the proofs in \cite{jezekEquationalTheoriesSemilattices1990} that $T$ is in a countably infinite language, so the question remains open in the intended setting of Jackson and Lee (i.e., monoids).

Note that this example also shows that the cardinality result of \Cref{TheoremVarieties} cannot be proved by constructing independent systems, thus ruling out the standard methods of Jankov formulas and their variants. Such approaches would fail to show that the variety corresponding to $T$ has $\conti$ many subvarieties. This further explains why we had to employ the toolkit of descriptive set theory.

\end{remark}

Next, we present an application of our method to logics and study the \emph{degrees of the finite model property} (\emph{degrees of fmp} for short). The degrees of fmp were introduced in \cite{bezhanishvili2023degreesfinitemodelproperty} as a modified version of the degrees of Kripke incompleteness \cite{fineIncompleteLogicContainingS41974} using finite Kripke frames. We first recall the definition for modal logics; the case of superintuitionistic logics is analogous. Let $L_0$ be a logic.

\begin{definition} \label{Def degree fmp}
    Let $\FFr$ be the class of all finite Kripke frames. For $L \in \NExt{L_0}$, let \[\FFr(L) = \{\F \in \FFr: \F \vDash L\},\] 
    and the \emph{fmp span} of $L$ (in $\NExt{L_0}$) be the set
    \[\fmp_{L_0}(L) = \{L' \in \NExt{L_0}: \FFr(L') = \FFr(L)\}.\] 
    The \emph{degree of fmp} of $L$ (in $\NExt{L_0}$) is the cardinality of the set $\fmp_{L_0}(L)$.
\end{definition}

The condition $\FFr(L) = \FFr(L')$ induces an equivalence relation on the lattice $\NExt{L_0}$, and the fmp span of $L$ refers to the equivalence class that $L$ belongs to. Thus, intuitively, the degree of fmp of $L$ measures the extent to which $L$ cannot be distinguished from other logics by the means of finite Kripke frames. It was proved in \cite{bezhanishvili2023degreesfinitemodelproperty} that every cardinal $0 < \kappa \leq \aleph_0$ or $\kappa = \conti$ is realized as the degree of fmp for superintuitionistic logics and transitive modal logics. So, the question naturally arises whether any cardinal between $\aleph_0$ and $\conti$ is also realized as the degree of fmp for some logic. We show that this cannot be the case.

\begin{lemma} \label{Lem degree fmp Pi02}
        For any logic $L \in \NExt{L_0}$, the set $\fmp_{L_0}(L)$ is $\Pi^0_2 (L_0 \oplus L)$. 
\end{lemma} 
\begin{proof}
    By \Cref{Lem interval Pi01}, there is a $\Pi^0_1$ formula $\alpha$ with the parameter $L_0$ that defines $\NExt{L_0}$. 
    
    A finite Kripke frame is a finite set with a binary relation. Given a finite Kripke frame $\F$ and a formula $\phi$, it is decidable whether $\F \models \phi$ by considering all possible valuations on $\F$ of propositional variables occurring in $\phi$. So, finite Kripke frames (up to isomorphism) can be recursively coded by natural numbers such that:  the validity relation 
    \begin{align*}
        \Val(f, i) \text{ iff } &\text{$f$ is the code of a finite Kripke frame $\F$ and} \\
         &\text{$i$ is the code of a formula $\phi$ and } \F \models \phi
    \end{align*}
    is recursive.

    For any $L \in \NExt{L_0}$ and $L' \subseteq \omega$, we have $L' \in \fmp_{L_0}(L)$ iff $L'$ satisfies $\alpha$ and the formula
    \[\beta = \forall f [f \in \FFr \to [(\forall i \in L' ~ \Val(f, i)) \leftrightarrow (\forall j \in L ~ \Val(f, j))]],\] 
    which is readily verified to be a $\Pi^0_2$ formula with the parameter $L$. Thus, the set $\fmp_{L_0}(L)$ is defined by $\alpha \land \beta$, which is a $\Pi^0_2$ formula with the parameters $L_0$ and $L$. Hence, the set $\fmp_{L_0}(L)$ is $\Pi^0_2(L_0 \oplus L)$.
\end{proof}

\begin{theorem} \label{Thm degree fmp}
    For any normal modal logic $L_0$ and $L \in \NExt{L_0}$, the degree of fmp of $L$ in $\NExt{L_0}$ is either $\leq \aleph_0$ or $2^{\aleph_0}$.
\end{theorem}

\begin{proof}
    This follows from \Cref{Lem degree fmp Pi02} analogously to the proof of \Cref{Lem interval}.
\end{proof}

\begin{remark} \label{remark degree si}
    It is straightforward to modify the proofs of \Cref{Lem degree fmp Pi02} and \Cref{Thm degree fmp} and prove similar results for superintuitionistic logics.
\end{remark}

\Cref{Thm degree fmp} shows that we can obtain complete antidichotomy results for the degrees of fmp without using the Continuum Hypothesis, answering the question posed in \cite[Section 8 (1)]{bezhanishvili2023degreesfinitemodelproperty}.

\begin{remark}
    It was observed in \cite{bezhanishvili2023degreesfinitemodelproperty} that for transitive modal and superintuitionistic logics, the fmp span is always an interval. So, in fact, \Cref{Lem interval}, applied to transitive modal and superintuitionistic logics, already answers question \cite[Section 8 (1)]{bezhanishvili2023degreesfinitemodelproperty}. We still provided a direct proof, since we believe it is more informative. Moreover, our result applies to all modal logics, whereas the fact that the fmp span is an interval was shown in \cite{bezhanishvili2023degreesfinitemodelproperty} only for transitive modal and superintuitionistic logics. 
\end{remark}

Finally, we apply our method to study classes of tense logics that are not intervals.  A tense logic $L$ is \emph{tabular} if it is characterized by a single finite Kripke frame, that is, there is a finite Kripke frame $\F$ such that for any formula $\phi$, $\phi \in L$ iff $\F \models \phi$. A tense logic is \emph{pretabular} if it is not tabular while all its proper consistent extensions are tabular. See \cite[Chapter 12]{czModalLogic1997} for details on tabular and pretabular logics. For a tense logic $L$, let $\sf{PTAB}(L)$ be the set of pretabular extensions of $L$.

Recently, Chen \cite{chen2025pretabulartenselogicss4t} showed that for each cardinal $\kappa \leq \aleph_0$ and $\kappa = \conti$, there is a tense logic $L$ extending $\sf{S4}_t$ such that $L$ has $\kappa$ many pretabular extensions. Given this kind of antidichotomy theorem, similar to the case of the degree of fmp, it is natural to ask if there is a tense logic that has $\kappa$ many pretabular extensions for $\aleph_0 < \kappa < \conti$. Contrary to intervals and fmp spans, by naively interpreting the definition of pretabularity, we would obtain a rather high upper bound for the complexity of $\sf{PTAB}(L)$, namely $\bm{\Pi}^1_1$, as the definition contains a universal quantification over logics. It is consistent with Zermelo-Frankel set theory that there exists a $\bm{\Pi}^1_1$ set $A$ whose cardinality satisfies $\aleph_0 < |A| < 2^{\aleph_0}$. For instance, this will hold in Cohen's \cite{Co63} original model for the negation of the Continuum Hypothesis, obtained by forcing over G\"odel's constructible universe $L$. Let $A$ be the set of all real numbers $x\in\mathbb{R}$ such that $x \in L_\alpha$ for some $x$-computable ordinal $\alpha$, where $L_\alpha$ denotes the $\alpha$th stage of G\"odel's constructible hierarchy. In this model, $A$ satisfies $|A| = \aleph_1 < \aleph_2 = 2^{\aleph_0}$. 

However, a more sophisticated characterization can bring the complexity down to $\bm{\Pi}^0_2$, where \Cref{fact Borel} applies. The following fact was proved in \cite[Lemma 11]{chagrovAlgorithmicAspectsPropositional1995} and \cite[Theorem 3.7]{chenTABULARITYPOSTCOMPLETENESSTENSE2024}.

\begin{fact} \label{fact tense tab}
    There is a set of formulas $\{\tab_n: n \in \omega\}$ such that for any consistent tense logic $L'$, we have that $L'$ is tabular iff $\tab_n \in L'$ for some $n \in \omega$. Moreover, the codes of $\tab_n$'s are computable from $n$.
\end{fact}

\begin{lemma} \label{Lem tense pretabular fin ax}
    For any tense logic $L$, we have that $L$ is pretabular iff it is not tabular while all its proper consistent \emph{finitely axiomatizable} extensions are tabular.
\end{lemma}

\begin{proof}
    It suffices to prove the right-to-left direction, as the other direction is clear by the definition. Assume for a contradiction that $L$ is not pretabular, $L$ is not tabular, and all its consistent finitely axiomatizable extensions are tabular. Then, $L$ has some proper consistent non-tabular extension $L'$, which is not finitely axiomatizable. So, there is some $\phi \in L'$ such that $\phi \notin L$. Since the logic $L + \phi$ is a proper consistent extension of $L$, it is tabular, and thus contains $\tab_n$ for some $n \in \omega$ by \Cref{fact tense tab}. It follows that $\tab_n \in L'$, which contradicts the assumption that $L'$ is not tabular. Therefore, we conclude the right-to-left direction of the statement.
\end{proof}

\begin{lemma} \label{Lem pretabular Pi02}
    For any tense logic $L$, the set $\sf{PTAB}(L)$ is $\Pi^0_2(L)$. 
\end{lemma}

\begin{proof}
    Similar to the proof of \Cref{Lem degree fmp Pi02}, we can code formulas in tense logic and finite Kripke frames so that the validity relation is recursive. We use the same notation $\Fml$ for the set of formulas. Also, by the proof of \Cref{Lem interval Pi01}, there is a $\Pi^0_1$ formula $\alpha(X)$ with the parameter $L$ that defines $\NExt{L}$. Let $f$ be the recursive function computing the code of $\tab_n$ in \Cref{fact tense tab} from $n$. 
    
    Tense logics have the same Hilbert-style proof system as bimodal logics, naturally extending that of unimodal logics (see, e.g., \cite[Section 3.6]{czModalLogic1997} and \cite[Section 1.6]{blackburnModalLogic2001}). Thus, proofs in tense logics can be coded in a standard way, yielding a recursive predicate $\sf{Proof}$ such that $\sf{Proof}(A, i, p, j)$ iff $p$ is the code of a proof of the formula $\phi_j$ in the logic $L_A + \phi_i$.

    By \Cref{Lem tense pretabular fin ax}, for any tense logic $L'$, we have that $L' \in \sf{PTAB(L)}$ iff $L' \in \NExt{L}$, $L'$ is not tabular, and for any formula $\phi$, either $L' + \phi = L'$, $L' + \phi$ is tabular, or $L' + \phi$ is the inconsistent logic, translating into the following formula:
    \[\alpha(L') \land \forall n (f(n) \notin L') \land \forall i \in \Fml [ i \in L' \lor \exists p \exists n  \sf{Proof}(L', i, p, f(n)) \lor \exists p \sf{Proof}(L', i, p, b)],\]
    where $b$ is the code of $\bot$. Since this formula is $\Pi^0_2$ with the parameter $L$, the set $\sf{PTAB}(L)$ is $\Pi^0_2(L)$. 
\end{proof}

\begin{theorem} \label{Thm tense pretab}
    For any tense logic $L$, the cardinality of the set of pretabular extensions of $L$ is either $\leq \aleph_0$ or $\conti$.
\end{theorem}

\begin{proof}
    This follows from \Cref{Lem pretabular Pi02} analogously to the proof of \Cref{Lem interval}.
\end{proof}

\section{Conclusion and future research} \label{Sec 4}

In this paper, we applied descriptive set theory to study the cardinality of classes of equational theories and logics via coding. The cardinality results were formulated for equational theories, varieties, modal logics, and tense logics to match each specific setting. However, the idea of coding applies to any formal system with a reasonable syntax such as \Cref{2: Thm equational theory}. For instance, the cardinality result of intervals (\Cref{Lem interval}) also holds for \emph{quasivarieties} and \emph{universal classes}. Moreover, \Cref{Thm tense pretab} holds for modal logics as well because of a characterization of tabular modal logics similar to \Cref{fact tense tab} (see, e.g., \cite[Corollary 12.2]{czModalLogic1997}. In practice, such results are of interest only if there is some sort of antidichotomy result for cardinals $\leq \aleph_0$ and $\conti$. 

A concrete open question is discussed in \Cref{remark open 1}, namely \cite[Question 6.4 (i)]{jacksonMonoidVarietiesExtreme2018} in the setting of finite languages, and specifically in the setting of monoids. We also mention two limitations of our method. An obvious one is that it heavily relies on coding. Therefore, the method works only if everything involved (e.g., equations, formulas, finite Kripke frames) can be effectively and uniquely coded. For example, this is not the case for Kripke frames, which form a proper class. We leave it for future research to apply this method to count Kripke complete logics or the \emph{degree of Kripke incompleteness} \cite{fineIncompleteLogicContainingS41974}. 

Another point is that the perfect set property is just one of the  \emph{regular properties} that Borel sets enjoy. Other regular properties include \emph{Lebesgue measurability} and the \emph{Baire property}. But these do not have immediate consequences for equational theories or logics. We leave it for future research to obtain results for classes of equational theories or logics, other than determining their cardinalities, by using the complexity of their corresponding sets of reals.

\subsection*{}
\begin{acknowledgements}
We are very grateful to George Metcalfe, Simon Santschi, and Niels Vooijs for pointing us to the open questions in \cite{jacksonMonoidVarietiesExtreme2018}. We would also like to thank Yurii Khomskii for valuable comments on an early draft of the paper. We are also grateful to the two anonymous reviewers for their careful reading and helpful suggestions. 
\end{acknowledgements}

\begin{funding}
    The first author would like to acknowledge the support of the Austrian Science Foundation (FWF) through grants 10.55776/STA139 and 10.55776/ESP3. The third author would like to acknowledge the support of the Student Exchange Support Program of the Japan Student Services Organization and the Student Award Scholarship of the Foundation for Dietary Scientific Research.
\end{funding}

\printbibliography[heading=bibintoc,title=References]

\end{document}